\documentclass[11pt,a4paper]{amsart}

\usepackage[english]{babel}
\usepackage{mathmio}
\title[Construction of HLS in $\CP^n$ and Hamiltonian stability]{Construction of homogeneous Lagrangian submanifolds in $\CP^n$ and Hamiltonian stability}
\author[D.~Petrecca \and F.~Podest\`a]{David Petrecca \and Fabio Podest\`a}
\address{Dipartimento di Matematica ``U. Dini''\endgraf Universit\`a degli Studi di Firenze\endgraf Viale Morgagni 67/A\endgraf Firenze, Italy}
\email{david.petrecca@gmail.com, podesta@math.unifi.it}

\begin{document}
\maketitle

\begin{abstract}
We apply the concept of castling transform of prehomogeneous vector spaces to produce new examples of minimal homogeneous Lagrangian submanifolds  in the complex projective space. Furthermore we verify the Hamiltonian stability of a low dimensional example that can be obtained in this way.
\end{abstract}

\section{Introduction}
Given a $2n$-dimensional K\"ahler manifold ($M,g,J$) with K\"ahler form $\omega$, a $n$-dimensional submanifold $L$ is said to be {\it Lagrangian\/} if the pull back of $\omega$ to $L$ vanishes. If there exists a Lie group $G$ of K\"ahler automorphisms of $M$ such that $L$ is a $G$-orbit, then $L$ is said to be a {\it homogeneous Lagrangian\/}. Such a class provides a large number of examples of Lagrangian submanifolds. \par
When $M = \CP^n$ and the group $G$ is compact and simple, a full classification of Lagrangian $G$-orbits has been obtained in \cite{art:bedulli}, while a full classification of homogeneous Lagrangian submanifolds of the quadrics has been achieved by Ma and Ohnita (\cite{MO}).  Our first result gives a way of producing new homogeneous Lagrangian submanifolds of the complex projective space starting from known ones. The construction is based on the main result of \cite{art:bedulli} and the castling transform, which will be explained in Section \ref{sec:pvs}, of a triple ($G,\rho, V$) consisting of a compact Lie group $G$, a complex vector space $V$ and a representation $\rho:G\to \GL(V)$.

\begin{thm}\label{cast} Let $(G,\rho,V)$ and $(G',\rho',V')$ be two triplets related by the castling transformation, where $G$ and $G'$ are two compact connected semisimple groups. Then the induced action of $G$ on $\P(V)$ admits a Lagrangian orbit if and only if the same holds for the $G'$-action on
$\P(V')$.
\end{thm}

In \cite{art:oh2}, Oh introduced the notion of Hamiltonian stability for minimal Lagrangian submanifolds of a K\"ahler manifold $(M,g,\omega)$. Given a minimal Lagrangian submanifold $\imath: L \to M$, it is said to be \emph{Hamiltonian stable} if the second variation of the volume functional through Hamiltonian variations is nonnegative, where Hamiltonian variations correspond to normal vector fields $V$ such that the one form $\imath^*(i_V\omega)$ is exact. Hamiltonian stability for Lagrangian submanifolds of the complex projective space turns out to be a strictly weaker condition than the usual stability, since e.g. the standard real projective space $\RP^n\subset \CP^n$ is minimal and Hamiltonian stable, but not stable in the usual sense. If we endow $\CP^n$ with the standard Fubini-Study metric $g_{\textup{FS}}$ with holomorphic sectional curvature $c$, then Oh (\cite{art:oh2}) proved that a minimal Lagrangian submanifold $L$ is stable if and only if the first eigenvalue $\lambda_1(L)$ for the Laplacian $\Delta$ relative to the induced metric and acting on $C^\infty(L)$ satisfies $\lambda_1(L)\geq \frac{n+1}{2}\ c$. Actually, since $\lambda_1(L)\leq \frac{n+1}{2}\ c$ for every minimal Lagrangian submanifold of $\CP^n$ by a result due to Ono (\cite{ono}), we see that stability is equivalent to $\lambda_1(L) = \frac{n+1}{2}\ c$.\par
It is a natural and interesting problem to classify all minimal, Hamiltonian stable Lagrangian submanifolds of $\CP^n$. In \cite{AO}, Amarzaya and Ohnita prove that every minimal Lagrangian submanifold with parallel second fundamental form is actually stable, while Bedulli and Gori (\cite{art:bedulli_stab}) and independently Ohnita (\cite{ohn}) exhibited the first example of a Hamiltonian stable Lagrangian submanifold which has non-parallel second fundamental form. This example sits inside $\CP^3$ and is homogeneous under the action of the group $\SU2$. Again using the castling transform, we are able to provide a new, low dimensional example,
\begin{thm}\label{ex} The group $G = \SU2\times \SU2$ acts in a standard way on $V = S^2(\C^2)\otimes \C^2 \cong \C^6$ and its induced action on $\CP^5$ has a minimal, Hamiltonian stable Lagrangian orbit $L$ with non-parallel second fundamental form. The fundamental group $\pi_1(L)$ is isomorphic to $\Z_4$.
\end{thm}
We remark that any Lagrangian orbit of a semisimple Lie group is minimal, whenever the ambient manifold is K\"ahler-Einstein (see \cite{art:bedulli}). We formulate the following
\begin{conj}If  a compact (semi)simple subgroup $G\subset \SU{N}$ for some $N$ admits a Lagrangian orbit $\mathcal O$ in $\CP^{N-1}$, then $\mathcal O$ is Hamiltonian stable. 
\end{conj}
In Section \ref{sec:pvs}, we prove Theorem \ref{cast}, while in Section \ref{sec:example} we prove the stability of our new example by using Oh's criterium and a direct computation of the first eigenvalue $\lambda_1(L)$.
\begin{notation} We use capital Latin letters for Lie groups and the corresponding lowercase Gothic letter for their Lie algebras. If $G$ is a group acting isometrically on the manifold $M$, for any $X \in \lie g$ we denote by $\hat X$ the induced Killing field on $M$.
\end{notation}
\section{Proof of Theorem \ref{cast}} \label{sec:pvs}
We first recall some notions that can be found in \cite{kimura,satokimura} and their application used in \cite{art:bedulli}.

Let $U$ a complex algebraic group, $V$ a complex vector space and $\rho$ a rational representation of $U$ on $V$. The triplet $(U,\rho,V)$ is said to be a \emph{prehomogeneous triplet (PVS)} if $V$ admits a Zariski-dense $U$-orbit $\Omega$. The isotropy subgroups of points in $\Omega$ are all conjugate to a subgroup $H \subseteq U$, which is called the \emph{generic isotropy} subgroup. The triplet is said to be irreducible if $\rho$ is.

Two triplets  $(U,\rho,V),(U',\rho',V')$ are said to be equivalent if there is a rational isomorphism $\phi: \rho(U) \rightarrow \rho'(U')$ and a linear isomorphism $\tau:V \rightarrow V'$ such that for all $g \in U$ we have $\tau \circ \rho(g) = \phi(\rho(g)) \circ \tau$.

We can now define the important notion of \emph{castling}. Two irreducible triplets $(U,\rho,V)$ and $(U',\rho',V')$ are  \emph{castling transform} of each other if there exists a third triplet  $(\tilde U,\tilde \rho,V^m)$ and a positive integer $m>n\geq 1$ such that
\begin{align*}
(U,\rho,V) & \cong (\tilde U \times \SL(n), \tilde \rho \otimes \Lambda_1, V^m \otimes V^n) \\
(U',\rho',V') & \cong (\tilde U \times \SL(m-n), \tilde \rho^* \otimes \Lambda_1,{V^m}^* \otimes V^{m-n}).
\end{align*}
A triplet is said to be \emph{reduced} if it is not a castling transform of any other triplet having a lower dimensional vector space.
It is also known that two castling-related prehomogeneous triplets have isomorphic generic isotropy subgroups (\cite[\S 2, Prop.~9]{satokimura}).

Given two compact connected groups $G,G'$ together with two irreducible representations $(\rho,V)$ and $(\rho',V)$ we say that the triplets $(G, \rho, V)$ and $(G',\rho',V')$ are castling related if the triplets $(G^\C,\rho,V)$ and $(G'^\C,\rho',V)$ are prehomogeneous and castling related in the sense explained above.

In order to prove Theorem \ref{cast}, we first prove a lemma which has its own interest.
\begin{lemma} \label{lemma:nolagr}
Let $G$ a compact connected semisimple Lie group acting linearly on some complex vector space endowed with the canonical symplectic structure. Then there are no Lagrangian $G$-orbits.
\end{lemma}

\begin{proof}
If $L$ is any $G$-orbit, the semisimplicity of $G$ implies that $\pi_1(L)$ is finite, by the long exact homotopy sequence. Therefore $H^1(L,\R) = 0$. On the other hand a classical result due to Gromov \cite{art:gromov} states that any compact Lagrangian submanifold of a complex vector space has nontrivial first cohomology group.
\end{proof}
We now have all the tools to give the
\begin{proof}[Proof of Theorem \ref{cast}]
Suppose that the $G$-orbit through $[p] \in \P(V)$ is Lagrangian. Then $G^\C \cdot [p]$ is open Stein by \cite{art:bedulli}. If $U = G^\C \times \GL(1)$ we claim that the orbit $U \cdot p$ is open Stein in $V$. In particular we claim that $\lie u_p = \lie g_{[p]}^\C$, which is reductive and therefore $U\cdot p$ is Stein by Matsushima's characterization \cite{matsu}.
Indeed
\[
\lie u_p	 = \Bigl \{ (X,z) \in \lie g^\C \oplus \C: Xp = -zp \Bigr \},
\]
in particular $X \in \lie g^\C_{[p]}$, hence $X \in (\lie g_{[p]})^\C$ because $G\cdot[p]$ is Lagrangian.
Now consider the orbit $G\cdot p \subset V$ and note that it is isotropic by a simple argument involving the expression of the moment map for actions in projective spaces (see, e.g., \cite{techniques}). By Lemma \ref{lemma:nolagr} it cannot be Lagrangian, so by dimensional reasons, it is a finite covering of the Lagrangian orbit in $\P(V)$. In particular $\lie g_p = \lie g_{[p]}$. So if $(X,z) \in \lie u_p$ if and only if $X \in \lie g_p^\C$ and $z=0$, therefore $\lie u_p = (\lie g_{[p]})^\C$ as we claimed. Furthermore, $U\cdot p$ is open for dimensional reasons.

Now we apply a castling transformation to get a triplet $(U',\rho',V')$, where $U' = G'^\C \times \GL(1)$. This triplet has generic isotropy isomorphic to the subgroup $H = U_p$, hence still reductive.

Let $\Omega = U'/H$ be the open Stein orbit in $V'$. This $U'$-orbit projects onto an open $U'$-orbit $\Omega'=  U'/H' \subset \P(V')$. In order to prove that $G'$ admits a Lagrangian orbit in $\P(V')$, we apply the main result in \cite{art:bedulli}, according to which it is enough to show that $\Omega'$ is Stein. Now, $\Omega'$ is Stein because $H'$ is reductive and this follows from standard arguments. Indeed we notice that $H \leq H'$ is normal and that $\dim_\C H'/H = 1$. By reductiveness we have $\lie h' = \lie h \oplus \lie m$, for some subspace $\lie m$ with $[\lie m,\lie h] \subset \lie m$.
Also $[\lie h,\lie m] \subset \lie h$ being $\lie h \subseteq \lie h'$ an ideal. Hence $[\lie h,\lie m]=0$ and $\lie m$ is a one-dimensional and central in $\lie h'$. Therefore $\lie h'$ is reductive as we claimed.
\end{proof}

\section{The Example and its stability} \label{sec:example}
Consider the group $G = \SU{2} \times \SU{2}$ acting on $V = S^2(\C^2) \otimes \C^2 \cong \C^6$ with the standard representation $\rho$. We consider the induced action on $\P(V) = \CP^5$.
Let $\{e_1,e_2 \}$ the standard basis of $\C^2$. We may define a unitary structure on $S^2(\C^2)$ with orthonormal basis given by $\{e_1^2, \sqrt 2 e_1 e_2, e_2^2 \}$ with respect to which the induced action of $\SU{2}$ becomes unitary. By tensoring with the standard basis of $\C^2$ we get an orthonormal basis of
$V$.
It is known that this action is Hamiltonian and that the moment map $\mu:\CP^5 \rightarrow \lie g^*$ has the form (see, e.g., again \cite{techniques})
\[
\mu([v])(X,Y) = - \frac i 2 \frac{\langle d\rho(X,Y)v,v \rangle}{\langle v,v\rangle},
\]
where $v \in V, (X,Y) \in \lie g = \lie{su}(2) \oplus \lie{su}(2)$.

We consider the point $p = \frac{1}{\sqrt 2}(e_1^2 \otimes e_1 + e_2^2 \otimes e_2) \in V$. A straightforward computation shows that $\mu([p]) = 0$ and, since $\lie g$ is semisimple, we conclude that the  $L:= G \cdot[p] \subset \CP^5$ is isotropic.

A direct computation shows that the isotropy subgroup $K := G_{[p]}$ is such that $\lie k = \R \cdot H$ where
\[
H = \biggl ( \begin{pmatrix}i & 0\\0 & -i\end{pmatrix},\begin{pmatrix}-2i & 0 \\ 0 & 2i \end{pmatrix} \biggr )
\]
and $K/K^o = \Z_4$, generated by the coset of the element
\[
\sigma = \biggl ( \begin{pmatrix}0&1\\-1& 0\end{pmatrix},\begin{pmatrix}0& i \\ i & 0 \end{pmatrix} \biggr ) \in K.
\]
By dimensional reasons $L$ is Lagrangian and moreover $\pi_1(L) = \Z_4$. Furthermore, being homogeneous under a semisimple group, the submanifold $L$ is also minimal by \cite{art:bedulli}. It is also clear that its second fundamental form is not parallel by the classification in \cite{art:naitak}.
\subsection{The metric on $L$}
We now compute explicitly the metric $g$ induced on $L$ by $g_{\textup{FS}}$. We denote with $B$ the Cartan-Killing form on $\lie g$ and we consider the $B$-orthonormal vectors of $\lie g$ given by
\begin{align*}
X_1& =(X,0)	& X_2 & =(0,X)	&
Y_1& =(Y,0)	& Y_2 & = (0,Y),
\end{align*}
where
\[
X = \frac{1}{\sqrt{8}} \begin{pmatrix} 0&1\\-1&0\end{pmatrix} \qquad Y = \frac{1}{\sqrt{8}} \begin{pmatrix}0 & i \\ i & 0 \end{pmatrix}.
\]
We also define the unitary vector
\[
V = \frac{1}{2\sqrt{10}} \biggl ( \begin{pmatrix}2i & 0 \\ 0 & -2i \end{pmatrix},\begin{pmatrix} -i & 0 \\ 0 & i\end{pmatrix} \biggr ).
\]
If we put $\lie m_j := \Span\{X_j,Y_j\}$ we have the $B$-orthogonal splitting
\[
\lie g = \lie k \oplus \R \cdot V \oplus \lie m_1 \oplus \lie m_2.
\]
We now compute the Killing fields at $p \in S^{11}$. We see that
\begin{align*}
\hat{X_1}_p 	& = \frac{1}{2 \sqrt 2} \bigl(-\sqrt 2 e_1 e_2 \otimes e_1 + \sqrt 2 e_1 e_2 \otimes e_2 \bigr); & \hat{X_2}_p & = \frac 1 4 \bigl (-e_1^2 \otimes e_2 + e_2^2 \otimes e_2 \bigr )\\
\hat{Y_1}_p & = \frac{i}{2 \sqrt 2} \bigl(\sqrt 2 e_1 e_2 \otimes e_1 + \sqrt 2 e_1 e_2 \otimes e_2 \bigr); & \hat{Y_2}_p & = \frac i 4 \bigl (e_1^2 \otimes e_2 + e_2^2 \otimes e_2 \bigr )
\end{align*}
and
\[
\hat{V}_p = \frac{i \sqrt5}{4} \bigr (e_1^2 \otimes e_1 - e_2^2 \otimes e_2 \bigl ).
\]
Starting from the Riemannian submersion $S^{11} \rightarrow \CP^5$ for the construction of the Fubini-Study metric $g_{\textup{FS}}$ with $c=4$ (\cite[vol.~2]{kn}) we compute their lengths with respect to the Riemannian metric $g$ on $L$:
\begin{align*}
\bigl \| \hat{X_1}_{[p]} \bigr \|_g 	 = \bigl \| \hat{Y_1}_{[p]} \bigr \|_g 	&=  \frac 1 2 &
\bigl \| \hat{X_2}_{[p]} \bigr \|_g	 = \bigl \| \hat{Y_2}_{[p]} \bigr \|_g	&= \frac{1}{2 \sqrt{2}}
\end{align*}
\[
\bigl \| \hat V_{[p]} \bigr \|_g =  \frac{\sqrt{5}}{2 \sqrt 2}.
\]

Define now
\[
V_1 = \frac{2 \sqrt 2}{\sqrt 5} V
\]
and
\begin{align*}
F_1 &= 2X_1 & F_2 &= 2 \sqrt 2 X_2 &
G_1 &= 2Y_1 & G_2 &= 2 \sqrt 2 Y_2.
\end{align*}
The metric $g$, induced on $G/K$ by the Fubini-Study metric on $\CP^5$, induces a metric $g_o$ on $\lie m := \R \cdot V \oplus \lie m_1 \oplus \lie m_2$. Note that these three submodules are mutually $\Ad(K)$-inequivalent and therefore mutually orthogonal and the vectors $V_1,F_1, F_2, G_1, G_2$ form a $g_o$-orthonormal basis.
\subsection{The Laplace operator on $C^\infty(L)$}
We claim that the first eigenvalue $\lambda_1(L)$ of the Laplacian $\Delta_g$ on $L$ is equal to the Einstein constant $\kappa = 12$ of $g_{\textup{FS}}$ on $\CP^5$. 

We now recall some general facts about invariant operators on homogeneous spaces. If $M^n = G/K$ is a homogeneous space and $\lie g = \lie k \oplus \lie m$ is an orthogonal splitting with respect to some $\Ad(G)$-invariant inner product on $\lie g$ we let $S(\lie m)$ the symmetric algebra of $\lie m$, $S(\lie m)_K^\C$ the complexification of the $\Ad(K)$-invariant subspace of $S(\lie m)$ and $\mathcal D(M)$ the space of $G$-invariant differential operators on $M$. In this notation we recall a well known result that can be found in \cite{helgason,mutour}.

\begin{thm}
Let ${Y_1,\ldots,Y_n}$ a basis of $\lie m$ and identify $S(\lie m)$ with polynomials in those indeterminates. Then the map $\hat \lambda	: S(\lie m)_K^\C	 \longrightarrow	\mathcal D(M)$ defined by
\[
P(Y_1,\ldots,Y_n)f(xK) = P \biggl (\dez{y_1},\ldots,\dez{y_n} \biggr )f \biggl ( x \exp \biggl( \sum_i y_i Y_i \biggr ) K \biggr )(0)
\]
is a linear isomorphism. Furthermore if ${Y_1,\ldots,Y_n}$ is an orthonormal basis with respect to an $\Ad(K)$-invariant scalar product $g_o$ on $\lie m$ and $\Delta_g$ is the Laplacian corresponding to the $G$-invariant metric $g$ on $M$ induced by $g_o$ , then
\[
\Delta_g = - \hat \lambda \biggl ( \sum_i Y_i^2 \biggr ).
\]
\end{thm}

Let $\rho: G \rightarrow \U(V)$ be a unitary representation of degree $d_\rho$ of the group $G$, let $V^K$ be the subspace of $V$ of vectors fixed by the subgroup $K$ where $m_\rho = \dim V^K$. A representation such that $m_\rho >0$ is said to be a \emph{spherical representation} of the pair $(G,K)$. Let $\{v_1,\ldots,v_{d_\rho}\}$ an orthonormal basis of $V$ such that the first $m_\rho$ elements are a basis of $V^K$. Define the functions on $G/K$ given by $\rho_{ij}(xK) = \langle \rho(x)v_j,v_i \rangle$ with $1 \leq j \leq m_\rho$ and $1 \leq i \leq d_\rho$.

The Peter-Weil Theorem (see e.g. \cite{helgason}) states that the set of functions $\{ \sqrt{d_\rho} \bar{\rho_{ij}} \}$, as $\rho$ varies among all spherical representations of $(G,K)$, is a complete orthonormal system of $L^2(M, \C)$ with respect to the standard $L^2$-norm corresponding to the $G$-invariant Riemannian metric $g$. 

We now classify all the spherical irreducible representations of our pair $(G,K)$. Any irreducible representation space is of the form $V_{k,m} = S^k(\C^2) \otimes S^m(\C^2)$ for some $k,m \in \mathbb N$. Since $\tau = (\id,-\id) \in K$ we see that $m$ must be even, say $m=2n$.
We have that
\begin{equation*}
H \cdot e_1^p e_2^{k-p} \otimes e_1^q e_2^{2n-q} =	 [(2p-k)i+4(n-q)i] e_1^p e_2^{k-p} \otimes e_1^q e_2^{2n-q}.
\end{equation*}
Computing also $ \sigma \cdot(e_1^p e_2^{2\ell-p} \otimes e_1^q e_2^{2n-q}) = (-1)^{n+p} e_1^{2\ell -p} e_2^p \otimes e_1^{2n-q}e_2^q $ we can conclude that
\begin{equation*}
V_{\ell,n}^{K}= \Span \biggl \{ v_{pq} := e_1^p e_2^{2\ell-p} \otimes e_1^q e_2^{2n-q} +(-1)^{n+p} e_1^{2\ell -p} e_2^p \otimes e_1^{2n-q}e_2^q \biggr \}
\end{equation*}
with the relations
\begin{equation} \label{relations}
p=\ell+2(n-q),	\qquad	0 \leq q \leq 2n,	\qquad	 2n-\ell \leq 2q \leq 2n+\ell.
\end{equation}
At this point we can compute the eigenvalues for the Laplace operator. Indeed we will explicitly write down the action of the operator
\[
D = d\rho(V_1^2) + d\rho(F_1^2)+d\rho(F_2^2) + d\rho(G_1^2) + d\rho(G_2^2)
\]
on the vectors $v_{pq} \in V_{\ell,n}^K$.

We have that, using the first equality in relations (\ref{relations}),
\[
d\rho(V_1)^2 \bigl(e_1^p e_2^{2\ell-p} \otimes e_1^q e_2^{2n-q}\bigr) = -4(q-n)^2 \bigl(e_1^p e_2^{2\ell-p} \otimes e_1^q e_2^{2n-q} \bigr).
\]
Also we compute
\begin{align*}
d\rho(F_1)^2 \cdot (e_1^p e_2^{2\ell-p} \otimes e_1^q e_2^{2n-q}) 	&= \frac 1 2 \biggl [p(p-1) e_1^{p-2}e_2^{2\ell-p+2} \\
																&-[p(2\ell-p+1)+(2\ell-p)(p+1)] e_1^p e_2^{2\ell-p} \\
																&+(2\ell-p)(2\ell-p-1)e_1^{p+2} e_2^{2\ell-p+2} \biggr] \otimes  e_1^q e_2^{2n-q}.
\end{align*}
In a similar way we compute
\[
\bigl (d\rho(F_1)^2 +d\rho(G_1)^2 \bigr ) \cdot \bigl (e_1^p e_2^{2\ell-p} \otimes e_1^q e_2^{2n-q} \bigr ) = 2(p^2- 2\ell p - \ell)\bigl (e_1^p e_2^{2\ell-p} \otimes e_1^q e_2^{2n-q} \bigr )
\]
and
\[
\bigl (d\rho(F_2)^2 +d\rho(G_2)^2 \bigr ) \cdot \bigl (e_1^p e_2^{2\ell-p} \otimes e_1^q e_2^{2n-q} \bigr ) = 4(q^2- 2 nq - n)\bigl (e_1^p e_2^{2\ell-p} \otimes e_1^q e_2^{2n-q} \bigr ).
\]

A direct check shows that the vectors $v_{pq} \in V_{\ell,n}^K$ are eigenvectors for the operator $D$, and therefore
\begin{align*}
-\Delta_g \bar{\rho_{pq,\alpha\beta}}(xK)	&= \bar{\langle \rho(x) D v_{pq},v_{\alpha \beta} \rangle} \\
										&={\lambda_{pq}} \bar{\rho_{pq,\alpha\beta}}(xK)\\
\end{align*}
with eigenvalue
\begin{align*}
\lambda_{pq}	&= 2(2(q-n)^2-(p^2-2\ell p -\ell) -2(q^2-2nq-n)) \\
			&= 2(2n^2+2n+\ell^2+\ell-(2q-2n)^2).\\
\end{align*}
For any natural numbers $\ell,n$ let $\mathcal F_{\ell,n}$ the set of pairs $(p,q)$ satisfying the relations in (\ref{relations}). Define
\[
\lambda_1^{\ell,n} := \min_{(p,q) \in \mathcal F_{\ell,n}} \lambda_{pq}
\]
so that the least eigenvalue for the Laplace operator is
\[
\lambda_1(L) = \min_{\ell,n}\lambda_1^{\ell,n},
\]
as $(\ell,n)$ varies among the natural numbers giving rise to a spherical representation of $(G,K)$.

Now note that $|2q -2n| \leq \ell$ and therefore $\lambda_1(L) \geq 2(2n^2+2n+\ell^2) \geq 24$ if $n \geq 2$, so we analyze the following cases
\begin{itemize}
\item If $n=0$ then $q=0$ and $p=\ell$ so $V_{\ell,0}^{K^o}$ is spanned by the vector $e_1^\ell e_2^\ell$ and this vector is fixed by $\sigma$ if and only if $\ell$ is even. Therefore $V_{\ell,0}$ is spherical only if $\ell \geq 2$ and this implies $\lambda_1(L) \geq 2(\ell+\ell^2) \geq 12$;

\item If $n=1$ and $\ell \geq 2$ then $\lambda_1(L) \geq 2(4+\ell^2) \geq 16$, so we can assume $\ell \leq 1$.
\begin{itemize}
\item If $\ell = 0$, then $V_{0,1}^{K^o}$ is spanned by $e_1 e_2$, but it is reversed by $\sigma$, so $V_{0,1}^K$ is trivial;
\item If $\ell = 1$, then $p = 1 + 2q - 2$ with $0 \leq q \leq 2$, hence $p = q = 1$. Then $V_{1,1}^K$ is spanned by $e_1 e_2 \otimes e_1 e_2$ and therefore $V_{1,1}$ is spherical and \[
\lambda_{11} = 2(4 + 1 + 1) = 12.
\]
\end{itemize}
\end{itemize}
So $\lambda_1(L)$ attains its lower bound which is equal to the Einstein constant $\kappa = 12$.

\bibliographystyle{amsplain}
\makeatletter
\renewcommand{\@biblabel}[1]{[#1]} 
\makeatother
\bibliography{biblio}

\providecommand{\bysame}{\leavevmode\hbox to3em{\hrulefill}\thinspace}
\providecommand{\MR}{\relax\ifhmode\unskip\space\fi MR }
\providecommand{\MRhref}[2]{%
  \href{http://www.ams.org/mathscinet-getitem?mr=#1}{#2}
}
\providecommand{\href}[2]{#2}
\begin{thebibliography}{10}

\bibitem{AO}
A.~Amarzaya and Y.~Ohnita, \emph{{H}amiltonian stability of certain minimal
  {L}agrangian submanifolds in complex projective spaces}, Tohoku Math. J.
  \textbf{55} (2003), no.~2, 583--610.

\bibitem{art:bedulli_stab}
L.~Bedulli and A.~Gori, \emph{A {H}amiltonian stable minimal {L}agrangian
  submanifold of projective space with non-parallel second fundamental form},
  Transform. Groups \textbf{12} (2007), no.~4, 611--617.

\bibitem{art:bedulli}
\bysame, \emph{Homogeneous {L}agrangian submanifolds}, Comm. Anal. Geom.
  \textbf{16} (2008), no.~3, 591--615.

\bibitem{art:gromov}
M.~Gromov, \emph{Pseudoholomorphic curves in symplectic manifolds}, Invent.
  Math. \textbf{82} (1985), no.~2, 307--347.

\bibitem{techniques}
V.~Guillemin and S.~Sternberg, \emph{Symplectic techniques in physics},
  Cambridge University Press, 1984.

\bibitem{helgason}
S.~Helgason, \emph{Differential geometry, {L}ie groups and symmetric spaces},
  Academic Press, 1978.

\bibitem{kimura}
T.~Kimura, \emph{A classification of prehomogeneous vector spaces of simple
  algebraic groups with scalar multiplications}, J. Algebra \textbf{83} (1983),
  72--100.

\bibitem{kn}
S.~Kobayashi and K.~Nomizu, \emph{Foundations of differential geometry},
  Interscience Publisher, 1963.

\bibitem{MO}
H.~Ma and Y.~Ohnita, \emph{On {L}agrangian submanifolds in complex
  hyperquadrics and isoparametric hypersurfaces in spheres}, Math. Z.
  \textbf{261} (2009), no.~4, 749--785.

\bibitem{matsu}
Y.~Matsushima, \emph{Espaces homogenes de {S}tein des groupes de {L}ie
  complexes}, Nagoya Math. J. \textbf{16} (1960), 205--218.

\bibitem{mutour}
H.~Mut\^o and H.~Urakawa, \emph{On the least positive eigenvalue of {L}aplacian
  for compact homogeneous spaces}, Osaka J. Math \textbf{17} (1980), 471--484.

\bibitem{art:naitak}
H.~Naitoh and M.~Takeuchi, \emph{Totally real submanifolds and symmetric
  bounded domains}, Osaka J. Math \textbf{19} (1982), 717--731.

\bibitem{art:oh2}
Y.-G. Oh, \emph{Second variation and stabilities of minimal {L}agrangian
  submanifolds in {K}\"ahler manifolds}, Invent. Math. \textbf{101} (1990),
  no.~2, 501--519.

\bibitem{ohn}
Y.~Ohnita, \emph{Stability and rigidity of special {L}agrangian cones over
  certain minimal {L}egendrian orbits}, Osaka J. Math \textbf{44} (2007),
  no.~2, 305--334.

\bibitem{ono}
H.~Ono, \emph{Minimal {L}agrangian submanifolds in adjoint orbits and upper
  bounds on the first eigenvalue of the {L}aplacian}, J. Math. Soc. Japan
  \textbf{55} (2003), 243--254.

\bibitem{satokimura}
M.~Sato and T.~Kimura, \emph{A classification of irreducible prehomogeneous
  vector spaces and their relative invariants}, Nagoya Math. J. \textbf{65}
  (1977), 1--155.

\end{thebibliography}
\end{document}